\documentclass[11pt]{article}
\usepackage{amsmath,amsfonts,amssymb, amsthm,enumerate,graphicx}
\usepackage{tikz,authblk}
\usepackage{subfig}
\usepackage{url}

\newcommand{\lk}[2]{{\rm lk}_{#1}(#2)}
\newcommand{\st}[2]{{\rm st}_{#1}(#2)}

\newtheorem{Lemma}{Lemma}[section]
\newtheorem{Theorem}[Lemma]{Theorem}
\newtheorem{Proposition}[Lemma]{Proposition}
\newtheorem{Corollary}[Lemma]{Corollary}

\newtheorem{definition}[Lemma]{Definition}
\newtheorem{Question}[Lemma]{Question}

\def\t{\tau}
\def\s{\sigma}

\def\p{\partial}

\def\D{\Delta}
\def\lk{\text{\textnormal{lk}}}
\def\st{\text{\textnormal{st}}}
\def\wrt{{\rm with respect to }}
\begin{document}

\title{Normal $4$-pseudomanifolds with a relative 2-skeleton}
\author{Biplab Basak, Mangaldeep Saha and Sourav Sarkar}
\date{}
\maketitle
\vspace{-10mm}
\noindent{\small Department of Mathematics, Indian Institute of Technology Delhi, New Delhi 110016, India.}

\begin{center}
\date{\today}
\end{center}

\hrule
\begin{abstract}
The study of face-number-related invariants in simplicial complexes is a central topic in combinatorial topology. Among these, the invariant $g_2$ plays a significant role. For a normal $d$-pseudomanifold $K$ ($d \geq 3$), it is known that $g_2(K) \geq g_2(\lk(v, K))$ for every vertex $v$. If $K$ has at most two singularities and satisfies $g_2(K) = g_2(\lk(t, K))$ for a singular vertex $t$, then $g_3(K) \geq g_3(\lk(t,K))$ holds. A normal $d$-pseudomanifold $K$ is called $g_2$- and $g_3$-optimal if $g_2(K) = g_2(\lk (t,K))$ and $g_3(K) = g_3(\lk (t,K))$ for a singular vertex $t$.

In this article, we establish structural results for normal $4$-pseudomanifolds under $g_2$- and $g_3$-optimality conditions. We show that if $K$ is a normal $4$-pseudomanifold with exactly one singular vertex $t$ and is $g_2$- and $g_3$-optimal at $t$, then $K$ can be obtained from boundary complexes of $5$-simplices through a sequence of operations of types vertex foldings and connected sums. When $K$ has exactly two singularities and is $g_2$- and $g_3$-optimal at one singular vertex, it is derived from the boundary complexes of $4$-simplices through a sequence of operations of types one-vertex suspensions, vertex foldings, and connected sums. Alternatively, we prove that if $K$ has two singular vertices and is $g_2$- and $g_3$-optimal at one of them, then it arises from boundary complexes of $5$-simplices through a sequence of operations of types vertex foldings, edge foldings, and connected sums.
\end{abstract}

\noindent {\small {\em MSC 2020\,:}  Primary 05E45; Secondary 05C30, 57Q15, 57Q25.
	
	}

\section{Introduction}
The \( g \)-conjecture \cite{McMullen1971} offers a detailed classification of \( f \)-vectors of simplicial polytopes and investigates which \( f \)-vectors could arise in other types of simplicial complexes. It was shown in \cite{Stanley} that every component of the \( g \)-vector associated with a simplicial \( d \)-polytope is non-negative. The conjecture’s sufficiency conditions for the \( f \)-vectors of simplicial polytopes and convex simplicial polytopes were rigorously established in \cite{Billera1980, Billera1981}. This naturally leads to broader questions about the geometric arrangement of general simplicial complexes in relation to their \( g \)-values. Notably, the condition \( g_1(\D) = 0 \) is satisfied if and only if \( \D \) is a simplex.  

Significant classification results have been derived by analyzing the third component, \( g_2 \), of the \( g \)-vector. Studies on \( g_2 \) of pseudomanifolds with isolated singularities are discussed in \cite{NovikSwartz2012}, while \cite{NovikSwartz2020} explores \( g \)-vectors of manifolds with boundary. In \cite{Swartz2008}, Swartz proved that the number of combinatorial manifolds, up to PL-homeomorphism, of a given dimension $d$ with an upper bound on $g_2$, is finite. Moreover, various structural results concerning homology manifolds and normal pseudomanifolds can be explored in \cite{Swartz2009}. The Lower Bound Theorem (LBT), formulated by Barnette \cite{Barnette1, Barnette2}, asserts that the boundary complex of a simplicial polytope of dimension  $(d+1)$ or, more generally, for any finite triangulation of a connected, compact \( d \)-manifold without boundary where \( d \geq 2 \), the inequality \( g_2(\D) \geq 0 \) holds. Walkup \cite{Walkup} further demonstrated that any triangulation of a 3-manifold with \( g_2 \leq 9 \) must be a triangulated 3-sphere, and when \( g_2 = 0 \), the structure is necessarily stacked. Moreover, the upper bound \( g_2 \leq 9 \) for triangulated 3-spheres is sharp, as certain triangulations of 3-dimensional sphere bundles (both orientable and non-orientable) achieve \( g_2 = 10 \). Additionally, the non-negativity of \( g_2(\D) \) has been verified for all normal pseudomanifolds of dimension at least three \cite{BagchiDatta, Fogelsanger, Kalai}. Furthermore, for \( d \geq 3 \), it has been established that the only normal \( d \)-pseudomanifolds for which \( g_2 = 0 \) are the stacked \( d \)-spheres (cf. \cite{Kalai, Tay}). 

Since Kalai’s work, significant efforts have been made to classify normal pseudomanifolds with small \( g_2 \).  Nevo and Novinsky classified such complexes for \( g_2 = 1 \) \cite{NevoNovinsky}, while Zheng extended the classification to \( g_2 = 2 \) \cite{Zheng}. In both cases, all resulting complexes are polytopal spheres.  

Basak and Swartz introduced the concept of a normal pseudomanifold \( K \) with relatively minimal \( g_2 \), meaning \( g_2(K) = g_2(\lk(v, K)) \) for some vertex \( v \) in \( K \) \cite{BasakSwartz}. In the same work, they introduced the notions of vertex folding and edge folding and characterized normal 3-pseudomanifolds with relatively minimal \( g_2 \) when \( K \) has at most two singularities. In particular, they classified normal 3-pseudomanifolds for \( g_2 = 3 \).  

Later, Basak and Gupta extended this classification to normal 3-pseudomanifolds with \( g_2 \leq 5 \) \cite{BasakGupta}. They proved that any such pseudomanifold is either a triangulation of a sphere or a suspension of \( \mathbb{RP}^2 \). Further progress was made by Basak, Gupta, and Sarkar, who provided a combinatorial characterization of normal 3-pseudomanifolds with relatively minimal \( g_2 \) when the complex has either three singular vertices, including one \( \mathbb{RP}^2 \)-singularity, or four singular vertices, including two \( \mathbb{RP}^2 \)-singularities \cite{BGS24}. In a subsequent work \cite{BGS23}, they characterized normal 3-pseudomanifolds \( K \) satisfying \( g_2(K) \leq g_2(\lk(v, K)) + 9 \) for some vertex \( v \), focusing on cases where \( K \) has either a single singularity or two singularities, at least one of which is an \( \mathbb{RP}^2 \)-singularity.  

In contrast, relatively little work has been done on normal \( d \)-pseudomanifolds for \( d \geq 4 \) in relation to \( g_2 \). Basak and Sarkar recently classified homology 4-manifolds \( K \) satisfying \( g_2(K) \leq 5 \) \cite{BS24}. In this paper, we focus on the combinatorial characterization of normal 4-pseudomanifolds with relatively minimal \( g_2 \).  

For a normal \( d \)-pseudomanifold \( K \) (\( d \geq 3 \)), it is known that \( g_2(K) \geq g_2(\lk(v, K)) \) for every vertex \( v \). Moreover, if \( K \) has at most two singularities and satisfies \( g_2(K) = g_2(\lk(t, K)) \) for a singular vertex \( t \), then the inequality \( g_3(K) \geq g_3(\lk(t, K)) \) holds. A normal \( d \)-pseudomanifold \( K \) is said to be \( g_2 \)- and \( g_3 \)-optimal if it satisfies \( g_2(K) = g_2(\lk(t, K)) \) and \( g_3(K) = g_3(\lk(t, K)) \) for some singular vertex \( t \).  In this paper, we establish structural results for normal 4-pseudomanifolds under \( g_2 \)- and \( g_3 \)-optimality conditions.  The main results of our article are the following.

\begin{Theorem}\label{dim 4 1-sing main}
Let  $K$ be a normal $4$-pseudomanifold with exactly one singular vertex $t$, and assume that $K$ is $g_2$- and $g_3$-optimal with respect to $t$. Then, $K$ is obtained from the boundary complexes of $5$-simplices by a sequence of operations consisting of vertex foldings and connected sums
\end{Theorem}

\begin{Theorem}\label{dim 4 2-sing main1}
Let $K$ be a normal $4$-pseudomanifold with exactly two singularities, and suppose that $K$ is $g_2$- and $g_3$-optimal with respect to a singular vertex. Then, $K$ is obtained from the boundary complex of $4$-simplices through a sequence of operations consisting of one-vertex suspensions, vertex foldings, and connected sums.
\end{Theorem}

\begin{Theorem}\label{dim 4 2-sing main2}
Let  $K$ be a normal $4$-pseudomanifold with exactly two singular vertices, and assume that $K$ is $g_2$- and $g_3$-optimal with respect to a singular vertex. Then, $K$ is obtained from some boundary complexes of $5$-simplices by a sequence of operations consisting of vertex foldings, edge foldings, and connected sums.
\end{Theorem}

\section{Preliminaries}
All simplices considered in this article are geometric, and all simplicial complexes are finite. We also assume that every simplicial complex contains $\emptyset$ as the only simplex of dimension $-1$. If $\sigma$ is an $n$-simplex in $\mathbb{R}^m$ for some $m$, which is the convex hull of $n+1$ affinely independent points $v_0, v_1, \dots, v_n$, then the simplex $\sigma$ is denoted by $v_0v_1\dots v_n$. These points $v_0, v_1, \dots, v_n$ are called the vertices of $\sigma$, and the set of all vertices of $\sigma$ is denoted by $V(\sigma)$. If a simplex $\tau$ is the convex hull of a subset ${v_0, \dots, v_i} \subseteq V(\sigma)$, then $\tau$ is called a face of $\sigma$, denoted by $\tau \leq \sigma$. The simplex $\tau$ may also be denoted by $\sigma - v_{i+1} \cdots v_n$. The boundary complex $\partial \sigma$ of a simplex $\sigma$ is defined as the collection of all proper faces of $\sigma$, i.e., $\{\t: \t\leq \s\hspace{.2cm}\text{and}\hspace{.2cm}\t\neq\s\}$.

If $\D$ is a simplicial complex, then its {\em geometric carrier} $|\D|$ is the union of all simplices in $\D$, equipped with the subspace topology induced from $\mathbb{R}^m$ for some $m$. A simplicial complex $\Delta$ is said to be pure if all its maximal simplices are of the same dimension. Every simplex in a simplicial complex is called a face of the complex, and the maximal faces are referred to as facets. The set of all vertices (or 0-simplices) of $\Delta$ is denoted by $V(\Delta)$. For $0\leq k\leq n$, the {\em $k$-skeleton} of $\D$, denoted by $Skel_k (\D)$, is the subcomplex consisting of all faces of dimensions up to $k$. For any subset $V'\subseteq V(\Delta)$, we define $\Delta[V']:=\{\sigma\in\Delta :V(\sigma)\subseteq V'\}$. A simplex $\sigma$ is called a {\em missing face} (or {\em missing simplex}) of $\Delta$ if $\partial \sigma\subseteq\Delta$ but $\sigma\notin\Delta$.

Two simplices $\sigma = u_0u_1\cdots u_k$ and $\tau = v_0v_1\cdots v_l$ in $\mathbb{R}^m$ for some $m\in \mathbb{N}$, are {\em independent} if $u_0,\dots ,u_k,v_0,\dots,v_l$ are affinely independent. In that case, $u_0\cdots u_kv_0\cdots v_l$ is a $(k + l + 1)$-simplex and is denoted by  $\sigma\star\tau$ or $\sigma\tau$. Two simplicial complexes $\D_1$ and $\D_2$ are said to be {\em independent} if $\s\t$ is a simplex of dimension $(i+j+1)$ for every pair of simplices $\s\in\D_1$ and $\t\in\D_2$ of dimensions $i$ and $j$, respectively. The join of two independent simplicial complexes $\D_1$ and $\D_2$ is the simplicial complex $\D_1\star\D_2:=\{\s\t : \s\in\D_1,\hspace{.15cm}\text{and}\hspace{.15cm}\t\in\D_2\}$. 
For a pair $(\sigma,\Delta)$, where $\sigma$ is a simplex and $\Delta$ is a simplicial complex, $\sigma\star\Delta$ denotes the simplicial complex $\{\alpha:\alpha\leq\sigma\}\star\Delta$. The {\em link} of a face $\sigma$ in $\Delta$ is the set $\{\gamma\in \Delta : \gamma\cap\sigma=\emptyset$ and $ \gamma\sigma\in \Delta\}$, denoted by $\lk (\sigma,\D)$. Similarly, the {\em star} of a face $\sigma$ in $\Delta$ is the set $\{\alpha\in \Delta : \alpha\leq\sigma \beta$ and $\beta\in \lk (\sigma,\D)\}$, denoted by $\st 
 (\sigma,\D$.

One of the most common enumerative tools of a $d$-dimensional simplicial complex $\D$ is its $f$-vector $(f_{-1}(\D),f_0(\D),\dots,f_d(\D))$, where $f_i(\D)$ is the number of $i$-dimensional simplices present in $\D$. We also define the $h$-vector of $\D$ as $(h_0(\D),h_1(\D),\dots,h_d(\D))$, where
$$ h_i(\D)=\sum_{j=0}^{i}(-1)^{i-j} \binom{d+1-j}{i-j}f_{j-1}(\D),$$
 and we define  $g_i(\D):= h_i(\D)-h_{i-1}(\D)$. In particular, $g_2(\D)=f_1(\D)-(d+1)f_0(\D) + \binom{d+2}{2}$ and $g_3(\D)= f_2(\D)-df_1(\D)+ \binom{d+1}{2}f_0 - \binom{d+2}{3}$.

 A strongly connected and pure simplicial complex $\D$ of dimension $d$ is called a {\em normal $d$-pseudomanifold} if every $(d-1)$-face of $\D$ is contained in exactly two facets in $\D$ and the links of all the simplices of dimension $\leq (d-2)$ are connected. If $t$ is a vertex in a normal pseudomanifold $\D$ whose link is not a triangulated sphere, then $t$ is called a {\em singular vertex} of $\D$, and in such cases, we say that $\D$ has a singularity at $t$.

 \begin{Proposition}{\rm \cite{Fogelsanger,Kalai}}\label{g2 of link is bounded by g2 of complex}
Let $d\geq 3$, and let $\D$ be a normal $d$-pseudomanifold. Then, $g_2(\lk(\s,\D))\leq g_2(\D)$ for every face $\s$ of co-dimension $3$ or more in $\D$.
\end{Proposition}


 A straightforward computation shows how the values of 
$g_i$ for $i=2,3$, change under the operations of handle addition and connected sum applied to a $d$-dimensional simplicial complex $K$. Specifically, the changes are given by the following expressions:

\begin{equation} \label{g_i: handles}
g_i(K^\psi) = g_i(K) + (-1)^{i}\binom{d+2}{i} \quad\text{for } i=2,3,
\end{equation}

\begin{equation} \label{g_2:connected sum}
g_i(K_1 ~\#_\psi~ K_2) = g_i(K_1) + g_i(K_2) \quad\text{for } i=2,3.
\end{equation}

\section{Implementation of Combinatorial operations}

 \begin{definition}{\rm \cite{BagchiDatta}}
 {\rm Let \( N \) be an induced subcomplex of a simplicial complex \( M \). One says that \( N \) is two-sided in \( M \) if \( |N| \) has a (tubular) neighbourhood in \( |M| \) homeomorphic to \( |N| \times [-1,1] \)   such that the image of \( |N| \) (under this homeomorphism) is \( |N| \times \{0\} \).}
 \end{definition}
  \begin{Proposition}{\rm \cite{BagchiDatta}}\label{2-sided lemma 3.3}
  Let $\D$ be a normal $d$-pseudomanifold with a missing $d$-simplex $\t$ such that $\p\t$ is two-sided in $\D$. Then $\D$ is the result of a connected sum or handle addition to normal $d$-pseudomanifolds.  
 \end{Proposition}

\begin{Proposition}\label{d-dim connected cum}
Let $\D$ be a normal $d$-pseudomanifold and $\t$ be a missing $d$-simplex in $\D$. If for every vertex $x$ in $\t$, $\p(\t-x)$ separates $\lk (x,\D)$, then $\D$ was formed using a handle addition or a connected sum.    
\end{Proposition}
 \begin{proof}
   In view of Proposition \ref{2-sided lemma 3.3}, it suffices to prove that $\p\t$ is two-sided in $\D$.  Let $x$ be a point in $|\p\t|$. If $x$ is not a vertex, then since $\D$ is a normal $d$-pseudomanifold, each small metric ball around $x$ in $|\D|$ is homeomorphic to the unit $d$-ball in $\mathbb{R}^d$. Moreover, $\p\t$ divides each metric ball around $x$ into two components. On the other hand, for every vertex $x$ in $\t$,  the fact that $\p(\t-x)$ separates $\lk (x,\D)$ implies that each small metric ball around each vertex of $\t$ in $|\D|$ is also separated into two components by $|\p\t|$. Since $|\p\t|$ is compact, gluing all such metric balls around the points of $|\p\t|$ provides us with a $\mathbb{Z}_2$ bundle over $|\p\t|$. Since $|\p\t|$ is a $(d-1)$-sphere, the bundle must be trivial. Hence, $|\p\t|$ divides a small neighborhood of itself in $|\D|$. Hence, $\p\t$ is two-sided in $\D$.
 \end{proof}
\begin{Proposition}[\cite{BasakSwartz}] \label{handle addition does not have g2-minimal}
 Let $d\geq 3$, and let $\D$ be a normal $d$-pseudomanifold. If $\D^{\psi}$ is obtained from $\D$ by a handle addition, then the normal $d$-pseudomanifold $\D^{\psi}$ cannot be $g_2$-minimal.    
 \end{Proposition}

Handle addition and connected sum are standard parts of combinatorial topology, but the operation of {\em  folding} was recently introduced in \cite{BasakSwartz}.

\begin{definition}[Vertex folding \cite{BasakSwartz}] 
{\rm 
Let $\sigma_1$ and $\sigma_2$ be two facets of a simplicial complex $K$, whose intersection is a single vertex $x.$  A bijection $\psi:\sigma_1 \to \sigma_2$ is {\em  vertex folding admissible} if $\psi(x) = x$ and for all other vertices $y$ of $\sigma_1$, the only path of length two from $y$ to $\psi(y)$ is $P(y, x, \psi(y)).$ For a vertex folding admissible map $\psi$, we can form the complex $K^\psi_x$ by identifying all faces $\rho_1 \leq \sigma_1 $ and $\rho_2 \leq \sigma_2 $, such that $\psi(\rho_1) = \rho_2,$ and then removing the facet formed by identifying $\sigma_1$ and $\sigma_2.$  In this case, we say that $K^\psi_x$ is a {\em  vertex folding} of $K$  at $x.$   In a similar spirit, $K$ is a {\em  vertex unfolding} of $K^\psi_x.$ }
\end{definition}

\noindent A straightforward computation shows that if $K^\psi_x$ is obtained from a $d$-dimensional  simplicial complex  $K$ by a vertex folding at $x$, then

 \begin{equation} \label{folding g2}
g_i(K^\psi_x) = g_i(K) + (-1)^{i}\binom{d+1}{i} \quad\text{for } i=2,3.
\end{equation}
The definition of {\em edge folding} follows the same pattern as vertex folding.
\begin{definition}[Edge folding \cite{BasakSwartz}]
{\rm 
Let $\sigma_1$ and $\sigma_2$ be two facets of a simplicial complex $K$, whose intersection is an edge $uv$. A bijection $\psi : \sigma_1 \to \sigma_2$ is {\em edge folding admissible} if $\psi(u)= u, \psi(v) = v$, and for all other vertices $y$ of $\sigma_1$, all paths of length two or less from $y$ to $\psi(y)$ pass through either $u$ or $v$. Identify all faces $\rho_1\leq \sigma_1$ and $\rho_2 \leq \sigma_2$, such that $\psi : \rho_1\to \rho_2$ is a bijection. The complex obtained by removing the facet resulting from identifying $\sigma_1$ and $\sigma_2$ is denoted by $K^\psi_{uv}$ and is called an {\em edge folding} of $K$ at $uv$. In a similar spirit,  $K$ is an {\em edge unfolding} of $K^\psi_{uv}$.}
\end{definition}
\noindent If $K$ is a normal $d$-pseudomanifold and $K^\psi_{uv}$ is obtained from $K$ by an edge folding at $uv$, then 
\begin{equation} \label{edge folding g2}
g_i(K^\psi_{uv}) = g_i(K) + (-1)^{i}\binom{d}{i} \quad\text{for } i=2,3.
\end{equation}

 \begin{Proposition}[\cite{BasakSwartz}] \label{3-dim vertex folding}
Let $K$ be a  normal $3$-pseudomanifold. Let $\t:=abcd$ be a missing tetrahedron in $K$ such that $(i)$ for $x\in\{b,c,d\}$, $\partial(\t-x)$ separates  $\lk (x,K)$ and $(ii)$  $\partial(bcd)$  does not separate  $\lk (a,K)$. Then there exists a  normal $3$-pseudomanifold $K'$ such that  $K = (K')^\psi_a$, i.e., $K $ is obtained from a vertex folding at $a \in K'$,  and $abcd$ is the image of the removed facet. 
\end{Proposition}

\begin{definition}[One-vertex suspension \cite{BagchiDatta98}]\label{one-vertex  suspension}
{ \rm Let $\D$ be a normal $(d-1)$-pseudomanifold and $v$ be a vertex in $\D$. Consider the  normal $d$-pseudomanifold $\sum_{v,u}\D:=(v\star\{\tau\in\D : v\nleq\tau\})\cup(u\star\D)$, where $u$ is a new vertex.  The complex $\sum_{v,u}\D$ is called an \textit{one-vertex suspension} of $\D$ \wrt the vertex $v$.}
\end{definition}

The combinatorial operations of one-vertex suspension, vertex folding and edge folding have played a key role in dimension 3 in the combinatorial description of $g_2$-minimal normal pseudomanifolds with at most two singularities known so far. In fact, appropriate applications of these operations, together with the connected sum on boundary complexes of a 4-simplex, can provide a specific combinatorial structure, as described below:

\begin{Proposition}{\rm \cite{BasakSwartz}}\label{3-dim npm two singularity}
Let $\D$ be a $g_2$-minimal normal $3$-pseudomanifold with at most two singular vertices, $t$ and $s$, such that $g_2(\lk (t,\D))\geq g_2(\lk (s,\D))$. Then $\D$ is obtained from a one-vertex suspension of $\lk (s,\D)$ and some boundary complexes of $4$-simplices by a sequence of operations of the form vertex foldings and connected sums. In particular, if $\D$ has exactly one singular vertex, then  $\D$ is obtained from some boundary complexes of $4$-simplices by a sequence of operations of the form vertex foldings and connected sums.
\end{Proposition}

\begin{Proposition}\label{4 dim 2-vertex two sided}
Let $\D$ be a normal $4$-pseudomanifold and $\s:=uvabc$ be a missing facet in $\D$ such that, for each $x\in V(\s)\setminus \{u,v\}$, $\p(\s-x)$ separates $\lk (x,\D)$ into two parts. Then the $2$-dimensional sphere $\p(\s-v)$ is two-sided in $\lk (v,\D)$.
\end{Proposition}
\begin{proof}
Let $\s=uvabc$, and let $\lk (uvbc,\D)=\{p_1,n_1\}, \lk (uvac,\D)=\{p_2,n_2\}, \lk (uvab,\D)=\{p_3,n_3\}, \lk (vabc,\D)=\{p_4,n_4\}$, and $\lk (uabc,\D)=\{p_5,n_5\}$. For each $x\in\{a,b,c\}$, $\p(\s-x)$ separates $\lk (x,\D)$ into two parts say $S_x^{+}$ and $S_x^{-}$, where $S_x^{+}$ and $S_x^{-}$ are $3$-dimensional normal pseudomanifolds with common boundary $\p(\s-x)$. Without loss of generality, assume that the vertices $p_2,p_3,p_4,$ and $p_5$ are in $S_{a}^{+}$. Then, $n_2,n_3,n_4$, and $n_5$ belong to $S_{a}^{-}$.

Since $\lk (cav,\D)$  is a circle and $\{p_4,n_4,p_2,n_2\}\subseteq V(\lk (cav,\D))$ with $p_4,p_2\in S_{a}^{+}$ and $n_4,n_2\in S_{a}^{-}$, it follows that $\lk (cav,\D)$ is the union of two distinct paths: $P_1:=P(b,p_4,\dots ,p_2,u)$ and  $P'_1:=P(b,n_4,\dots ,n_2,u)$, intersecting along the vertices $b$ and $u$. Similarly, $\lk (bav,\D)$ is the union of two distinct paths: $P_2:=P(c,p_4,\dots ,p_3,u)$ and $P_2':=P(c,n_4,\dots ,n_3,u)$; and $\lk (uav,\D)$ is the union of  two distinct paths: $P_3:=P(b,p_3,\dots ,p_2,c)$ and $ P_3':=P(b,n_3,\dots ,n_2,c)$, intersecting at the vertices in $\{c,u\}$ and $\{b,c\}$, respectively. Moreover, the paths $P_1,P_2$ and $P_3$ belong to $S_{a}^{+}$, while the other three paths, $P_1',P_2'$ and $P_3'$ belong to $S_{a}^{-}$.

Observe that the path $P_2$ belongs to $\lk (b,\D)$, and since $\p(\s-b)$ separates $\lk (b,\D)$ into two parts, $S_{b}^{+}$ and $S_{b}^{-}$, the path $P_2$ is entirely contained in either $S_{b}^{+}$ or $S_{b}^{-}$. Without loss of generality, we assume that $P_2$, i.e., $P(c,p_4,\dots ,p_3,u)$ is in $S_{b}^{+}$. Then, $p_4,p_3\in S_{b}^{+}$ and $n_4,n_3\in S_{b}^{-}$. Now, from the hypothesis in the first paragraph, $\lk (b,\D)$ contains two distinct paths: $P_4:=P(v,p_4,\dots , p_5,u)$ and $P_4':= P(v,n_4,\dots , n_5,u)$, which intersect along the vertices in $\{v,u\}$. Moreover, $p_4\in S_{b}^{+}$ implies that $p_5\in S_{b}^{+}$, and $n_4\in S_{b}^{-}$ implies that $n_5\in S_{b}^{-}$. Without loss of generality, assume that  $p_1\in S_{b}^{+}$ and $n_1\in S_{b}^{-}$. Then, we have $\{p_1,p_4,p_3,p_5\}\subseteq S_{b}^{+}$ and $\{n_1,n_4,n_3,n_5\}\subseteq S_{b}^{-}$.

Since $\lk (cbv,\D)$ is a circle and $p_1,n_1,p_4,n_4\in\lk (cbv,\D)$ with $p_1,p_4\in S_{b}^{+}$ and $n_1,n_4\in S_{b}^{-}$, it follows that $\lk (cbv,\D)$ is a circle that is the union of two distinct paths: $P_5:=P(u,p_1,\dots,p_4,a)$ and $P_5':=P(u,n_1,\dots,n_4,a)$, intersecting along the vertices in $\{a,u\}$. Similarly, $\lk (ubv,\D)$ is the union of two distinct paths: $P_6:=P(c,p_1,\dots,p_3,a)$ and $ P_6':=P(c,n_1,\dots,n_3,a)$, intersecting along the vertices in $\{a,c\}$.

\begin{figure}[ht]
\tikzstyle{ver}=[]
\tikzstyle{vertex}=[circle, draw, fill=black!100, inner sep=0pt, minimum width=4pt]
\tikzstyle{edge} = [draw,thick,-]
\centering
\begin{tikzpicture}[scale=1]
\begin{scope}[shift={(0,0)}]
\foreach \x/\y/\z in {-1.5/0/c,2.9/0/v,0/1.7/b, -0.6/-1.8/d, 0.2/.95/p4,-.6/.3/p5,1/.3/p3,-.15/-.7/p2}{\node[vertex] (\z) at (\x,\y){};}
\foreach \x/\y/\z in {-1.7/0/c,3.1/0/v,0/2/b, -0.7/-2.1/u, 0.3/.7/p_4,-.3/.2/p_5,1/.1/p_3,-.15/-.9/p_2}{
\node[ver] () at (\x,\y){$\z$};}

\foreach \x/\y in {v/b,v/c,v/d,b/c,d/c,b/d}{\path[edge] (\x) -- (\y);}
\foreach \x/\y in {p3/p4,p4/p2,p4/p5,p3/p2,p3/p5,p5/p2}{\draw[densely dotted] (\x) -- (\y);}
\foreach \x/\y in {v/p3,v/p4,v/p2,d/p2,d/p5,d/p3,c/p2,c/p5,c/p4,b/p5,b/p3,b/p4}{\draw[loosely dotted] (\x) -- (\y);}


\end{scope}

\begin{scope}[shift={(6,0)}]
\foreach \x/\y/\z in {-1.5/0/c,2.9/0/v,0/1.7/b, -0.6/-1.8/d, 0.2/.95/p4,-.6/.3/p5,1/.3/p3,-.15/-.7/p2}{\node[vertex] (\z) at (\x,\y){};}
\foreach \x/\y/\z in {-1.7/0/c,3.1/0/v,0/2/b, -0.7/-2.1/u, 0.3/.7/n_4,-.3/.2/n_5,1/.1/n_3,-.15/-.9/n_2}{
\node[ver] () at (\x,\y){$\z$};}

\foreach \x/\y in {v/b,v/c,v/d,b/c,d/c,b/d}{\path[edge] (\x) -- (\y);}
\foreach \x/\y in {p3/p4,p4/p2,p4/p5,p3/p2,p3/p5,p5/p2}{\draw[densely dotted] (\x) -- (\y);}
\foreach \x/\y in {v/p3,v/p4,v/p2,d/p2,d/p5,d/p3,c/p2,c/p5,c/p4,b/p5,b/p3,b/p4}{\draw[loosely dotted] (\x) -- (\y);}


\end{scope}
%
%
%
\end{tikzpicture}
\caption{$\lk (a,\D)$} \label{fig:1}
\end{figure}

Using similar arguments as in the last two paragraphs, and assuming that $P_5$ is in $S_{c}^{+}$, we have $p_1, p_4\in S_{c}^{+}$ and $n_1, n_4\in S_{c}^{-}$. Now, from the hypothesis in the first paragraph, it follows that $\lk (dcv,\D)$ is a circle, which is the union of two distinct paths: $P_7:=P(b,p_1,\dots, p_2, a)$ and $ P'_7:= P(b,n_1,\dots,n_2, a)$. Moreover, $p_1\in  S_{c}^{+}$ implies that $P_7$ is contained in $S_{c}^{+}$, and $P'_7$ is contained in $S_{c}^{-}$.

\begin{figure}[ht]
\tikzstyle{ver}=[]
\tikzstyle{vertex}=[circle, draw, fill=black!100, inner sep=0pt, minimum width=4pt]
\tikzstyle{edge} = [draw,thick,-]
\centering
\begin{tikzpicture}[scale=1]
\begin{scope}[shift={(0,0)}]
\foreach \x/\y/\z in {-1.5/0/c,2.9/0/v,0/1.7/b, -0.6/-1.8/d, 0.2/.95/p4,-.6/.3/p5,1/.3/p3,-.15/-.7/p2}{\node[vertex] (\z) at (\x,\y){};}
\foreach \x/\y/\z in {-1.7/0/c,3.1/0/a,0/2/b, -0.7/-2.1/u, 0.3/.7/p_4,-.3/.2/p_1,1/.1/p_3,-.15/-.9/p_2}{
\node[ver] () at (\x,\y){$\z$};}

\foreach \x/\y in {v/b,v/c,v/d,b/c,d/c,b/d}{\path[edge] (\x) -- (\y);}
\foreach \x/\y in {p3/p4,p4/p2,p4/p5,p3/p2,p3/p5,p5/p2}{\draw[densely dotted] (\x) -- (\y);}
\foreach \x/\y in {v/p3,v/p4,v/p2,d/p2,d/p5,d/p3,c/p2,c/p5,c/p4,b/p5,b/p3,b/p4}{\draw[loosely dotted] (\x) -- (\y);}


\end{scope}

\begin{scope}[shift={(6,0)}]
\foreach \x/\y/\z in {-1.5/0/c,2.9/0/v,0/1.7/b, -0.6/-1.8/d, 0.2/.95/p4,-.6/.3/p5,1/.3/p3,-.15/-.7/p2}{\node[vertex] (\z) at (\x,\y){};}
\foreach \x/\y/\z in {-1.7/0/c,3.1/0/a,0/2/b, -0.7/-2.1/u, 0.3/.7/n_4,-.3/.2/n_1,1/.1/n_3,-.15/-.9/n_2}{
\node[ver] () at (\x,\y){$\z$};}

\foreach \x/\y in {v/b,v/c,v/d,b/c,d/c,b/d}{\path[edge] (\x) -- (\y);}
\foreach \x/\y in {p3/p4,p4/p2,p4/p5,p3/p2,p3/p5,p5/p2}{\draw[densely dotted] (\x) -- (\y);}
\foreach \x/\y in {v/p3,v/p4,v/p2,d/p2,d/p5,d/p3,c/p2,c/p5,c/p4,b/p5,b/p3,b/p4}{\draw[loosely dotted] (\x) -- (\y);}


\end{scope}
%
%
%
\end{tikzpicture}
\caption{$\lk (v,\D)$} \label{fig:2}
\end{figure} 

Now, there are six pairs of paths, namely:
\begin{enumerate}[$(i)$]

 \item $P_1$ and  $P'_1$ in $\lk (cav,\D)$,
 
 \item  $P_2$ and $P'_2$ in $\lk (bav,\D)$,
 
 \item  $P_3$ and $ P'_3$ in $\lk (uav,\D)$,
 
 \item $P_5$ and $P'_5$ in $\lk (cbv,\D)$,
 
 \item  $P_6$ and $ P'_6$ in $\lk (ubv,\D)$, and

  \item $P_7$ and $ P'_7$ in $\lk (cuv,\D)$,

 \end{enumerate}
 where each path in a pair intersects at exactly two vertices, and the intersecting vertices belong to $\{u, a,b,c\}$. Moreover, a pair of paths $P_i$ and $P'_j$ may intersect, but in such cases, one can take an appropriate subdivision of an edge to obtain a two-sided neighborhood. Therefore, $\p(\s-v)$ is two-sided in $\lk (v,\D)$.
\end{proof}   
\begin{Corollary}\label{4 dim two sided}
Let $\D$ be a normal $4$-pseudomanifold and $\s$ be a missing facet in $\D$ such that, for each $x\in V(\s)\setminus\{v\}$, $\p(\s-x)$ separates $\lk (x,\D)$ into two parts. Then, the $2$-dimensional sphere $\p(\s-v)$ is two sided in $\lk (v,\D)$.
\end{Corollary}

\begin{Lemma}\label{4-dim vertex folding}
Let $\D$ be a normal $4$-pseudomanifold and $\tau:=vabcd$ be a missing facet in $\D$ such that, for each $x\in V(\tau)\setminus\{v\}$, $\p(\tau-x)$ separates $\lk (x,\D)$ into two parts, but $\p(\tau-v)$ does not separate $\lk (v,\D)$ into two parts. Then $\D$ is obtained from a normal $4$-pseudomanifold $\D'$ by a vertex folding at the vertex $v$.
\end{Lemma}

\begin{proof}
 Since $abcd \in \D$, its link consists of two vertices, say $p$ and $n$. Given that, for every vertex $x \in\{a,b,c, d\}$, the boundary complex $\partial(\tau- x)$ separates $\lk (x,\D)$ into two parts, say $S_{x}^{+}$ and $S_{x}^{-}$, where $S_{x}^{+}$ and $S_{x}^{-}$ are normal $3$-pseudomanifolds with boundary, having the common boundary complex $\partial(\tau - x)$. Moreover, we assume that $ p \in S_{x}^{+}$ and $n \in S_{x}^{-}$  for every vertex $x \in\{a,b,c, d\}$. Let us denote the sets of vertices $V(S_{x}^{+})\setminus V(\t)$ and $V(S_{x}^{-})\setminus V(\t)$ by $V_{x}^{+}$ and $V_{x}^{-}$, respectively. Finally, let $E_{x}^{ \pm}=\left\{w x: w \in V_{x}^{ \pm}\right\}$, and define $E^{ \pm}=E_{a}^{ \pm} \cup E_{b}^{ \pm} \cup E_{c}^{ \pm} \cup E_{d}^{ \pm}$. By definition, $E^{+} \cap E^{-}=\emptyset$.

Let $\sigma \nleq abcd$ be a simplex in $\Delta$ that does not contain the vertex $v$, but has a non-empty intersection with $abcd$. Let $E_{\sigma}$ be the set of edges of $\sigma$ that have exactly one endpoint in $abcd$. Then $E_{\sigma} \subseteq E^{+} \cup E^{-}$. In order to construct $\Delta'$, we aim to show that $E_{\sigma} \subset E^{+}$ or $E_{\sigma} \subset E^{-}$. 

If a face $\sigma$ intersects $abcd$, then for some $x \in\{a, b, c, d\}$, we have $ \sigma \in \st (x,\D)$. Therefore, $V(\sigma) \setminus\{a, b, c, d\} \subset V_{x}^{+}$ or $V_{x}^{-}$. We need to show that if $\s$ is a simplex in $\D$ such that $xy\leq \s$, where $x$ and $y$ are two vertices belonging to $\{a,b,c,d\}$, then $V(\sigma) \setminus\{a, b, c, d\} \subset V_{x}^{+}$ implies $V(\sigma) \setminus\{a, b, c, d\} \subset V_{y}^{+}$. To prove this, let $\sigma = \gamma \star xy $ be such a simplex in $\Delta$, where  $V(\sigma) \setminus \{a, b, c, d\} \subset V_{x}^{+}$. Then $\gamma \in \lk (xy,\D)$. Since $\p(\t-x)$ separates $\lk (x,\D)$ into two parts, the boundary complex $\p(\t-xy)$ also separates $\lk (xy,\D)$ into two parts, which are exclusively contained in $S_x^{+}$ and $S_x^{-}$, respectively. Furthermore, $\lk (abcd,\D)=\{p,n\}$ implies that one of the portions contains the vertex $p$, and the other portion contains the vertex $n$. Since $V(\sigma) \setminus\{a, b, c, d\} \subset V_{x}^{+}$, both simplices $p\star (abcd-xy)$ and $\gamma$ belong to the portion of $\lk (xy,\D)$ contained in $S_x^{+}$. On the other hand, the fact that $\lk (y,\D)$ is separated into two parts by $\p(\t-y)$ implies that one of the two portions of $\lk (xy,\D)$, while $\lk (xy,\D)$ is separated into two parts by $\p(\t-xy)$, is exclusively contained in $S_y^{+}$ and the other portion is exclusively contained in $S_y^{-}$. Moreover, both simplices $p\star (abcd-xy)$ and $\gamma$ belong to the same portion of $\lk (xy,\D)$, while separated by $\p(\t-xy)$. However, since $p\in S_y^{+}$, we have $p\star (abcd-xy)\in S_y^{+}$. Therefore, the simplex $\gamma$ also lie in the portion of $\lk (xy,\D)$ contained in $S_y^{+}$. Hence, $\gamma\in S_y^{+}$, and we conclude that $V(\sigma) \setminus\{a, b, c, d\} \subset V_{y}^{+}$. Therefore, $E_{\sigma}$ is a subset of exactly one of $E^{+}$ and $E^{-}$.

 Let $\tilde{\Delta}= \D[V(\Delta) \setminus\{v\}]$. Then, $\tilde{\Delta}$ is a normal $4$-pseudomanifold with boundary $\lk (v,\D)$. In our construction of $\Delta'$, we define an intermediate complex $\tilde{\Delta}'$, which is closely related to $\tilde{\Delta}$ and has vertex set $V(\tilde{\Delta}) \cup \{a',b', c', d'\}$, where $a',b', c'$, and $d'$ are four new vertices. Let $\tilde{\sigma}$ be a simplex in $\tilde{\Delta}$ that is not a face of $abcd$. We define $\tilde{\sigma}'$ to be $\tilde{\sigma}$ if $\tilde{\sigma}$ does not intersect $abcd$ or $E_{\tilde{\sigma}} \subset E^{+}$. If $E_{\tilde{\sigma}} \subset E^{-}$, then $\tilde{\sigma}'$ is defined to be $\tilde{\sigma}$ with all occurrences of $a, b, c$, and $d$ replaced by $a', b', c'$, and $d'$, respectively. Now define $\tilde{\Delta}'$ to be the simplicial complex whose facets are $\{\tilde{\sigma}': \tilde{\sigma} \hspace{.2cm}\text{ is a facet of}\hspace{.2cm} \tilde{\Delta}\} \cup \{na' 
 b'c'd'\}$.

 The $3$-simplex $a b c d$ of $\tilde{\Delta}'$ is contained in exactly one facet, namely $p a b c d$. Similarly, the $3$-simplex $a' b' c' d'$ is contained in exactly one facet, namely $n a' b' c' d'$. All other $3$-simplices $\tilde{\sigma}'$ of $\tilde{\Delta}'$ are contained in one or two facets in $\tilde{\Delta}'$, depending on whether $\tilde{\sigma}$ was contained in one or two facets in $\tilde{\Delta}$.

 By Corollary \ref{4 dim two sided}, the 2-dimensional sphere $\p(abcd)$ is two-sided in $\lk (v,\D)$. Since $\partial(a b c d)$ does not separate $\lk (v,\D)$, by Lemma 3 of \cite{BagchiDatta}, $\lk v$ was formed via handle addition. Let $\lk (v,\D)=S^{\psi}$, where $S$ is a normal $3$-pseudomanifold and the image of the domain of $\psi$ is $\{a,b,c,d\}$. Hence $\tilde{\Delta}'$ is a pseudomanifold with boundary $S$. Finally, set $\Delta'$ to be $\tilde{\Delta}'$ with its boundary coned off by $v$. Therefore, $\Delta'$ is a pseudomanifold and $\lk (v,\Delta')=S$. Furthermore, from the construction, $\Delta=\left(\Delta'\right)^{\psi}_{v}$, where $\psi:\left\{v, a', b', c', d'\right\} \rightarrow\{v, a, b, c, d\}$ is an admissible map for vertex folding. Two vertices, say $x$ and $y$, in $\Delta'$ can be connected by a path by concatenating the lift of a path in $\Delta$ from $x$ to $v$, and then from $v$ to $y$. Thus $\Delta'$ is connected. To finish the proof we need to show that the link of every vertex $\D'$, except $v$ (that has already been shown), is a normal $3$-pseudomanifold.

 Let $x \in\{a,b, c, d\}$. Since $\lk (x,\D)= \lk (x,\Delta') \#{ }_{\psi} \lk (x',\Delta')$, and $\lk (x,\D)$ is a  normal $3$-pseudomanifold, each of the links $\lk (x,\Delta')$ and $\lk (x',\Delta')$ must be a normal $3$-pseudomanifold. Finally, suppose that $w$ is a vertex of $\Delta'$ other than those in $\{v, a, b, c, d, a',$ $ b', c',d'\}$. Then the link of $w$ in $\Delta'$ is the same as the link of $w$ in $\Delta$, except that for each vertex $x\in \{a, b, c, d\}$ with $w x \in E^{-}$, $x$ is replaced by $x'$ in the link of $w$ in $\Delta'$. Thus, the link of $w$ in $\Delta'$ is a normal $3$-pseudomanifold, and hence $\Delta'$ is a normal $4$-pseudomanifold.
\end{proof}

\begin{Lemma}\label{4-dim edge folding}
Let $\D$ be a normal $4$-pseudomanifold, and let $\s=uvabc$ be a missing facet in $\D$ such that, for each $x\in\{a,b,c\}$, $\p(\s-x)$ separates $\lk (x,\D)$ into two parts, but for each $y\in \{u,v\}$, $\p(\s-y)$ does not separate $\lk (y,\D)$ into two parts. Further, assume that $\lk (uv,\D)$ is not separated into two parts by $\partial(abc)$. Then $\D$ is obtained by an edge folding from a normal $4$-pseudomanifold along the edge $uv$.
\end{Lemma}

\begin{proof}
 Let $\lk (uvbc,\D)=\{p_1,n_1\}, \lk (uvac,\D)=\{p_2,n_2\}, \lk (uvab,\D)=\{p_3,n_3\}, \lk (vabc,\D)$ $=\{p_4,n_4\}$, and $\lk (uabc,\D)=\{p_5,n_5\}$. For $x\in\{a,b,c\}$, $\p(\s-x)$ separates $\lk (x,\D)$ into two parts, say $S_x^{+}$ and $S_x^{-}$, where $S_x^{+}$ and $S_x^{-}$ are normal $3$-pseudomanifolds with common boundary $\p(\s-x)$. We choose the vertices $p_i$'s and $n_i$'s as in the proof of Proposition \ref{4 dim 2-vertex two sided}. In particular, for $x\in\{a,b,c\}$, we assume $p_i\in S_x^{+}$ and $n_i\in S_x^{-}$ for $1\leq i\leq 5$. Let us denote the sets $V(S_{x}^{+})\setminus V(\s)$ and $V(S_{x}^{-})\setminus V(\s)$ by $V_{x}^{+}$ and $V_{x}^{-}$, respectively.

 We define a new complex $\D'$, whose vertex set is given by $V(\D')=(V(\D)\setminus\{a,b,c\})\sqcup \{a^{+},b^{+},c^{+},a^{-},b^{-},c^{-}\}$. The collection of facets of $\D'$ is the set $\{\t':\t\text{ is a 4-simplex in }\D\} \cup \{uva^{+}b^{+}c^{+}, uva^{-}b^{-}c^{-}\}$, where $\t'$ is defined as follows:

 \begin{enumerate}[$(i)$]
     \item If $V(\t)\cap V(\s)=\emptyset$, then we take $\t'=\t$.

     \item If $V(\t)\cap V(\s)=\{a\}$, then the vertices of $\s$ other than $a$ lie either in $V_a^{+}$ or in $V_a^{-}$. We define $\t'=(\t-a)\star a^{+}$ (resp. $\t'=(\t-a)\star a^{-}$) if $V(\t)\setminus\{a\}\subseteq V_a^{+}$ (resp. $V(\t)\setminus\{a\}\subseteq V_a^{-}$). A similar construction for $\t'$ is considered when $V(\t)\cap V(\s)$ is $\{b\}$ or $\{c\}$.

      \item If $V(\t)\cap V(\s)=\{a,u\}$, then the other three vertices of $\t$ lie either in $V_a^{+}$ or in $V_a^{-}$. Therefore, we consider $\t'=(\t-a)\star a^{+}$ (resp. $\t'=(\t-a)\star a^{-}$) if $V(\t)\setminus\{a\}\subseteq V_a^{+}$ (resp. $V(\t)\setminus\{a\}\subseteq V_a^{-}$). A similar situation holds when $V(\t)\cap V(\s)$ is $\{a,v\}, \{b,u\}, \{ b,v\}, \{c,u\}$, or $\{c,v\}$.

      \item If $V(\t)\cap V(\s)=\{a,u,v\}$, then we take $\t'$ as the same as described in $(ii)$ and $(iii)$.

     \item If $V(\t)\cap V(\s)=\{a,b\}$, then the other vertices of $\t$ lie either in $V_a^{+}$ or in $V_a^{-}$. Using the same argument as in Lemma \ref{4-dim vertex folding}, the vertices of $\t$ other than $a$ and $b$ belong to $V_b^{+}$ (resp.  $V_b^{-}$) if these vertices belong to $V_a^{+}$ (resp. $V_a^{-}$). Therefore, we consider $\t'=(\t-ab)\star a^{+}b^{+}$ (resp. $\t'=(\t-ab)\star a^{-}b^{-}$) if $V(\t)\setminus\{a,b\}\subseteq V_a^{+}$ (resp. $V(\t)\setminus\{a,b\}\subseteq V_a^{-}$). Similar constructions hold when $V(\t)\cap V(\s)$ is $\{a,c\}$ or $\{b,c\}$. 

    \item If $V(\t)\cap V(\s)=\{a,b,u\}$, then by the same argument as in the previous case, the vertices of $\t$ other than $a,b$ and $u$ lie either in $V_a^{+}\cap V_b^{+}$ or in $V_a^{-}\cap V_b^{-}$. Similarly as in the previous case, we take $\t'$ as $(\t-ab)\star a^{+}b^{+}$ and $(\t-ab)\star a^{-}b^{-}$ when the vertices of  $\t$ other than $a,b$, and $u$ lie in $V_a^{+}\cap V_b^{+}$ and $V_a^{-}\cap V_b^{-}$, respectively. The other cases are similar to this.

    \item If $V(\t)\cap V(\s)=\{a,b,u,v\}$, then we take $\t'$ as the same as described in $(v)$ and $(vi)$.

    \item If $V(\t)\cap V(\s)=\{a,b,c\}$, then the other vertices of $\t$ lie either in $V_a^{+}\cap V_b^{+}\cap V_c^{+}$ or in $V_a^{-}\cap V_b^{-}\cap V_c^{-}$. In the first case, we consider $\t'$ as $(\t-abc)\star a^{+}b^{+}c^{+}$, and in the second case, we take $\t'$ as $(\t-abc)\star a^{-}b^{-}c^{-}$.

   \item If $V(\t)\cap V(\s)=\{a,b,c,u\}$, then the structure of $\t'$ will be the same as described in $(viii)$. The same construction will be taken if $V(\t)\cap V(\s)=\{a,b,c,v\}$.     
 \end{enumerate}

Consider a bijection $\psi: \D'\to \D'$ that sends $(u,v,a^{+}, b^{+}, c^{+})$ to $(u,v,a^{-}, b^{-}, c^{-})$. Then $\psi$ is admissible for edge folding at the edge $uv$, and $\D$ is obtained from $\D'$ by identifying all faces $\rho_1\leq uva^{+}b^{+}c^{+}$ and $\rho_2\leq uva^{-}b^{-}c^{-}$ such that $\psi(\rho_1)=\rho_2$, and then removing the identified facet. Note that the vertices $a,b$, and $c$ are the vertices in $\D$ obtained by identifying $a^{+}$ with $a^{-}$, $b^{+}$ with $b^{-}$, and $c^{+}$ with $c^{-}$, respectively. Hence, $\D=(\D')_{uv}^{\psi}$. It remains to prove that $\D'$ is a normal $4$-pseudomanifold.

We prove that the link of every vertex in $\D'$ is a normal $3$-pseudomanifold. Observe that $\lk (a,\D)= \lk (a^{+},\D')\#_{\psi} \lk (a^{-},\D')$. Since $\D$ is a normal $4$-pseudomanifold, $\lk (a,\D)$ is a normal 3-pseudomanifold, and hence both $\lk (a^{+},\D')$ and $\lk (a^{-},\D')$ must be normal $3$-pseudomanifolds. A similar description applies for the links of $b^{+}, b^{-}, c^{+}$, and $c^{-}$. 

Let $y$ be a vertex in $\D'$ other than $a^+, a^-,b^+,b^-,c^+$, and $c^-$. Then the link of $y$ in $\D'$ is the same as the link of $y$ in $\D$, except for each $x\in\{a,b,c\}$ if $y\in S_x^+$ (resp. $y\in S_x^-$) in $\D$, then $x$ is replaced by $x^+$ (resp. $x^-$) in the link of $y$ in $\D'$. Hence, the link of $y$ in $\D'$ is a normal 3-pseudomanifold. It remains to prove that $\lk (u,\D')$ and $\lk (v,\D')$ are normal 3-pseudomanifolds.

Since, for each $x\in V(uvabc)\setminus \{u,v\}$, $\p(\s-x)$ separates $\lk (x,\D)$ into two parts, it follows from Proposition \ref{4 dim 2-vertex two sided} that the $2$-dimensional sphere $\p(abcv)$ is two-sided in $\lk (u,\D)$. Furthermore, since  $\lk (uv,\D)$ is not separated into two parts by $\partial(abc)$, we arrive at the following scenario: $\tau:=vabc$ is a missing 3-simplex in $\lk (u,\D)$,  $\lk (v,\lk (u,\D))$ is not separated into two parts by $\partial(abc)$, and for each $x\in \{a,b,c\}$,  $\lk (x, \lk (u,\D))$  is separated into two parts by $\partial(\tau-x)$.  Consequently, by Proposition \ref{3-dim vertex folding}, $\lk (u,\D)$ is obtained from $\lk (u,\D') u$  through a vertex folding at $v$. As a result,  $\lk (u,\D')$ is a normal 3-pseudomanifold. By a similar argument,  $\lk (v,\D')$ is also a normal 3-pseudomanifold. This completes the proof.
\end{proof}

\section{$g_2$- and $g_3$-optimal normal pseudomanifolds}
\begin{Proposition}{\rm \cite{BasakSwartz}}\label{vertex outside star}
Let $\D$ be a normal $d$-pseudomanifold such that $g_2(\D)=g_2(\lk (t,\D))$. If a vertex $v$ is not in $\st (t,\D)$, then $\lk (v,\D)$ is a stacked sphere.
\end{Proposition}
\begin{Lemma}\label{G(D)=G(st t)}
Let $\D$ be a normal $d$-pseudomanifold, and let $t$ be a singular vertex in $\D$ such that $g_2(\D)=g_2(\lk (t,\D))$.  If $\D$ contains at most two singular vertices, then there exists a normal $d$-pseudomanifold $\D'$ that contains the vertex $t$ such that, $Skel_{1}(\D')=Skel_{1}(\st (t,\D'))$ and $\D$ is obtained from $\D'$ by (possibly zero) facet subdivisions.   
\end{Lemma}
\begin{proof}
 We prove the result by induction on the number of vertices that are not in $\st (t,\D)$. For the base case of zero, if $\D$ contains an edge, say $e$, that is not in $\st (t,\D)$, then $G(\st (t,\D)) \cup \{e\}$ is generically $(d+1)$-rigid. Since $g_2(\D[V(\st (t,\D))])=g_2(\D)+1$, we obtain a contradiction. Therefore, $G(\D)=G(\st (t,\D))$.

 For the induction step, assume that $u$ is a vertex that is not in $\st (t,\D)$. By Lemma \ref{vertex outside star}, $\lk (u,\D)$ is a stacked sphere. If $\lk (u,\D)$ is the boundary of a $d$-simplex, say $\s$, then $\s\notin\D$, and we replace $\st (u,\D)$ with the simplex $\s$. Since the resulting complex has fewer vertices than $\D$, we are done by the induction hypothesis.
 
 Now, assume that $\lk (u,\D)$ is a stacked sphere that is not the boundary of a $d$-simplex. Then $\lk (u,\D)$ contains a missing $d-1$-simplex, say $\t$. We first observe that $\t$ is a face of $\D$; otherwise, we can retriangulate $\st (u,\D)$ by removing $u$ and its incident faces, inserting $\t$, and then conning off the spheres obtained by $\lk (u,\D)$ and $\t$. If $\D_1$ is the resulting complex in such instances, then $g_2(\D_1)=g_2(\D)-1$, and $\D_1$ would still contain a vertex whose link is isomorphic to the link of $t$, which is a contradiction. Therefore, $\t$ is in $\D$ and we get $u \star\t$ as a missing facet in $\D$. 

 Since $V(\D)\setminus\{t\}$ contains at most one singular vertex and $u\notin \st (t,\D)$, at most one vertex of $\t$ can be singular. Further, by Proposition \ref{d-dim connected cum} and Lemma \ref{4-dim vertex folding}, $\D$ was obtained by one of the following ways: either by a connected sum of two normal pseudomanifolds or by a vertex folding from a normal 4-pseudomanifold. If $\D$ was obtained by a vertex folding from a normal pseudomanifold $\D_2$, then $g_2(\D_2)=g_2(\D)-\binom{d+1}{2}$ and $\D_2$ contains the vertex $t$ with $g_2(\st (t,\D_2))=g_2(\st (t,\D))=g_2(\D)$, which is a contradiction. Therefore, $\D$ was obtained by a connected sum of two normal pseudomanifolds $\D_4$ and $\D_5$. Further, exactly one of $\D_4$ and $\D_5$ contains the vertex $t$. Let $t\in\D_4$. Then $\lk (t,\D_4)$ and $\lk (t,\D)$ are isomorphic, and hence $g_2(\D_4)=g_2(\D)$. Therefore, $g_2(\D_5)=0$ and we conclude that $\D_4$ is a stacked sphere. This completes the proof.
 \end{proof}

\begin{Lemma}\label{appearence of g_3}
Let $K$ be a $g_2$-minimal normal $d$-pseudomanifold with respect to a singular vertex $t$. If $K$ contains at most two singular vertices, then $g_3(K)\geq g_3(\st (t,K))$.
\end{Lemma}
\begin{proof}
Since $K$ is a $g_2$-minimal normal $d$-pseudomanifold with at most two singularities, Lemma \ref{G(D)=G(st t)} implies the existence of a normal $d$-pseudomanifold $\D$ that contains the vertex $t$ and satisfies $Skel_{1}(\D)=Skel_{1}(\st (t,\D))$. Moreover, $K$ is obtained from $\D$ by (possibly zero) facet subdivisions. Consequently, we have $f_0(\D)=f_0(\st (t,\D))$ and $f_1(\D)=f_1(\st (t,\D))$. Now, 
  \begin{eqnarray*}
g_3(\D)&=& f_2(\D)-df_1(\D)+\binom{d+1}{2}f_0(\D)-\binom{d+2}{3}\\
&=& f_2(\D)-d f_1(\st (t,\D))+\binom{d+1}{2}f_0(\st (t,\D))-\binom{d+2}{3}\\
&=& f_2(\D)-f_2(\st (t,\D))+g_3(\st (t,\D))\\
&\geq& g_3(\st (t,\D)).
\end{eqnarray*}
Since $g_3(K)=g_3(\D)$ and $g_3(\st (t,K))=g_3(\st (t,\D))$, it follows from the above relation that $g_3(K)\geq g_3(\st (t,K))$. This completes the proof.
\end{proof}

Let \( K \) be a normal \( d \)-pseudomanifold with at most two singular vertices. It is known that for any vertex \( v \), the inequality \( g_2(K) \geq g_2(\lk (v,K)) \) holds. Suppose \( K \) is \( g_2 \)-minimal with respect to a singular vertex \( t \), meaning that \( g_2(K) = g_2(\lk (t,K)) \). Then, it follows that \( g_2(\lk (t,K)) \geq g_2(\lk (v,K)) \) for all vertices \( v \). Furthermore, by Lemma \ref{appearence of g_3}, we obtain the bound \( g_3(K) \geq g_3(\st (t,K)) \). We say that \( K \) is {\em \( g_2 \)- and \( g_3 \)-optimal} if it satisfies \( g_2(K) = g_2(\lk (t,K)) \) and \( g_3(K) = g_3(\lk (t,K)) \). 
Notably, the condition \( g_2(K) = g_2(\lk (t,K)) \) ensures that \( g_2(\lk (t,K)) \geq g_2(\lk (v,K)) \) for all vertices \( v \). However, the equality \( g_3(K) = g_3(\lk (t,K)) \) does not necessarily imply a similar relation, as it may happen that \( g_3(K) = g_3(\lk (t,K)) \) is negative while \( g_3(\lk (v,K)) \) is non-negative. For instance, if \( \lk (v,K) \) is a stacked sphere, then \( g_3(\lk (v,K)) = 0 \).

Before we move further, observe that if a $d$-manifold $K$ is $g_2$-minimal \wrt a vertex $t$, then by the same arguments as in the proof of Lemma \ref{G(D)=G(st t)}, we obtain a $d$-manifold $\D$ such that $Skel_{1}(\D)=Skel_{1}(\st (t,\D))$, and $K$ is obtained from $\D$ by (possibly zero) facet subdivisions. Moreover, $g_3(\D)\geq g_3(\st (t,\D))$ holds. Therefore, using the same notion as in the normal $d$-pseudomanifolds, the $g_2$- and $g_3$-optimality for manifolds can also be defined. We say that a $d$-manifold $K$ is $g_2$- and $g_3$-optimal if it satisfies \( g_2(K) = g_2(\lk (t,K)) \) and \( g_3(K) = g_3(\lk (t,K)) \).

Our main result is focused on normal pseudomanifolds that are both $g_2$- and $g_3$-optimal with respect to a singular vertex. Interestingly, this hypothesis remains invariant under certain well-known combinatorial operations. These operations are stated in Proposition \ref{Proposition invariant}, and the proofs follow directly from the definitions of the corresponding combinatorial operations. 

 \begin{Proposition}\label{Proposition invariant} Let $d\geq 3$. Then we have the following:
   \begin{enumerate}[$(i)$]
       \item The boundary complex of a $d$-simplex is $g_2$- and $g_3$-optimal with respect to every vertex.

       \item If $\D$ is a normal $d$-pseudomanifold and $\D'$ is obtained from $\D$ by a finite sequence of facet subdivisions, then $\D$ is $g_2$- and $g_3$-optimal if and only if $\D'$ is $g_2$- and $g_3$-optimal. In particular, stacked $d$-spheres are $g_2$- and $g_3$-optimal

       \item Let $\D=\D_1\#\D_2$, where $\D_1$ and $\D_2$ are normal $d$-pseudomanifolds that are $g_2$- and $g_3$-optimal. Suppose $t$ is a vertex in $\D$ that is not an identified vertex during the connected sum of $\D_1$ and $\D_2$. Then $t$ belongs to exactly one of $\D_1$ or $\D_2$. If $t\in\D_1$, then $\D$ is $g_2$- and $g_3$-optimal with respect to $t$ if and only if  $\D_1$ is $g_2$- and $g_3$-optimal with respect to $t$, and $\D_2$ is a stacked sphere. On the other hand, if $t$ is an identified vertex during the connected sum of $\D_1$ and $\D_2$, representing $t_1$ of $\D_1$ and $\psi(t_1)$ of $\D_2$, then $\D$ is $g_2$- and $g_3$-optimal with respect to $t$ if and only if  both $\D_1$ and $\D_2$ are $g_2$- and $g_3$-optimal with respect to $t_1$ and $\psi(t_1)$, respectively.

       \item Let $\D$ be a normal $d$-pseudomanifold, and let $\D^{\psi}$ be obtained by attaching a $d$-handle to $\D$. Then $\D^{\psi}$ cannot be $g_2$- and $g_3$-optimal with respect to any vertex in it. 

       \item Let $\D$ be a normal $d$-pseudomanifold, and let $\D^{\psi}_t$ be obtained from $\D$ by a vertex folding at the vertex $x$. Then $\D$ is $g_2$- and $g_3$-optimal with respect to $t$ if and only if  $\D^{\psi}_t$ is $g_2$- and $g_3$-optimal with respect to $t$. 

        \item Let $\D$ be a normal $d$-pseudomanifold, and let $\D^{\psi}_{uv}$ be obtained from $\D$ by an edge folding along the edge $uv$. Then $\D$ is $g_2$- and $g_3$-optimal with respect to $u$ if and only if  $\D^{\psi}_{uv}$ is $g_2$- and $g_3$-optimal with respect to $u$. 
   \end{enumerate}  
 \end{Proposition}

\begin{Lemma}\label{2-skeletn same}
Let $K$ be a normal $d$-pseudomanifold that is $g_2$- and $g_3$-optimal with respect to a singular vertex $t$. If $K$ contains at most two singular vertices, then there exists a normal $d$-pseudomanifold $\D$ containing the vertex $t$ such that $Skel_{2}(\D)=Skel_{2}(\st (t,\D))$, and $K$ is obtained from $\D$ by (possibly zero) facet subdivisions.   
\end{Lemma}
\begin{proof}
  Since $g_2(K)=g_2(\lk (t,K))$, from Lemma \ref{G(D)=G(st t)}, it follows that  there exists a normal $d$-pseudomanifold $\D$ containing the vertex $t$ such that $Skel_{1}(\D)=Skel_{1}(\st (t,\D))$, and $K$ is obtained from $\D$ by (possibly zero) facet subdivisions. Note that if $K$ and $\D$ differ by finitely many facet subdivisions, then $g_i(\D)=g_i(K)$ for all $i$. Therefore, $g_2(\D)=g_2(\st (t,\D))$ and $g_3(\D)=g_3(\st (t,\D))$.  Since $f_0(\D)=f_0(\st (t,\D))$ and $f_1(\D)=f_1(\st (t,\D))$, from $g_3(\D)=g_3(\st (t,\D))$, we have $f_2(\D)=f_2(\st (t,\D))$.  Therefore, from $\st (t,\D)\subseteq \D$, we conclude that $Skel_{2}(\D)=Skel_{2}(\st (t,\D))$.
  \end{proof}

Note that even if we have a $g_2$- and $g_3$-optimal $d$-manifold $K$, using the same proof as in Lemma \ref{2-skeletn same}, we can conclude the following.

\begin{Corollary}\label{Cor 2-skeletn same}
  Let $K$ be a $d$-manifold that is $g_2$- and $g_3$-optimal with respect to a vertex $t$. Then there exists a $d$-manifold $\D$ containing the vertex $t$ such that $Skel_{2}(\D)=Skel_{2}(\st (t,\D))$, and $K$ is obtained from $\D$ by (possibly zero) facet subdivisions. 
\end{Corollary}

\begin{Lemma}\label{4-manifold}
Let $K$ be a $4$-manifold. If $K$ is $g_2$- and $g_3$-optimal with respect to a vertex, then $K$ is a stacked sphere.
\end{Lemma}
\begin{proof}
Let $K$ be $g_2$- and $g_3$-optimal with respect to a vertex $t$. Using Corollary \ref{Cor 2-skeletn same}, we obtain a 4-manifold $\D$ and a vertex $t$ in $\D$ such that $Skel_{2}(\D)=Skel_{2}(\st (t,\D))$, and $K$ is obtained from $\D$ through finitely many facet subdivisions. It suffices to prove that $\D$ is a stacked sphere.

Let $\D'$ be the induced complex on $V(\D)\setminus \{t\}$. Then $\D'$ is a manifold with boundary. If $\D'$ does not contain an interior $3$-simplex, then $\D$ is the boundary complex of a 5-simplex. Now, assume that $\t$ is an interior $3$-simplex in $\D'$. Since every 2-simplex of $\D$ is also in $\st (t,\D)$, it follows that $\p\t\subseteq \st (t,\D)$. Thus, $t\star\t$ is a missing 4-simplex in $\D$. Therefore, by Proposition \ref{d-dim connected cum}, we have $\D=\D_1\#_\psi\D_2$, where $t$ is an identified vertex during the connected sum of $\D_1$ and $\D_2$, representing $t_1$ of $\D_1$ and $t_2$ of $\D_2$. Furthermore, both $\D_1$ and $\D_2$ are $g_2$- and $g_3$-optimal with respect to $t_1$ and $\psi(t_1)$, respectively. Note that either both $\D_1$ and $\D_2$ have a smaller $g_2$ than $\D$, or one of $\D_1$ and $\D_2$ has the same $g_2$ as $\D$ while the other one is a stacked sphere. 

Thus, every $4$-manifold $\D$ such that $Skel_{2}(\D)=Skel_{2}(\st (t,\D))$, where $\D[V(\D)\setminus\{t\}]$ contains an interior 3-simplex, can be written as a connected sum of two $g_2$- and $g_3$-optimal 4-manifolds, each having either a smaller $g_2$ value or fewer vertices than $\D$. Now, we proceed with these 4-manifolds. Since the number of vertices in $\D$ is finite, after a finite number of steps we obtain a $4$-manifold $\D''$ and a vertex $t''$ in $\D''$ such that $Skel_{2}(\D'')=Skel_{2}(\st (t'',\D''))$, where $\D[V(\D'')\setminus\{t''\}]$ contains no interior 3-simplex. Hence, $\D''$ is the boundary complex of a $5$-simplex. This implies that $\D$ is obtained by a sequence of connected sum operations from a collection of boundary complexes of 5-simplices. Hence $\D$ is a stacked sphere. This completes the proof.
\end{proof}

We now have the necessary tools to establish one of our main results, Theorem \ref{dim 4 1-sing main}.

\bigskip

\noindent {\em Proof of Theorem} \ref{dim 4 1-sing main}. 
Using Lemma \ref{2-skeletn same}, we have a normal pseudomanifold $\D$ and a vertex $t$ in $\D$ such that $Skel_{2}(\D)=Skel_{2}(\st (t,\D))$ and $K$ is obtained from $\D$ by a finite number of facet subdivisions. Furthermore, $t$ is a singular vertex of $\D$. We will use induction on $g_2$ and the number of vertices in $\D$, as needed. 

The base case is the minimal triangulation of a 4-handle with the boundary coned off. Consider a minimal triangulation of a 4-stacked ball along with an admissible bijection between a pair of boundary 3-simplices, and identify such a pair of 3-simplices in the boundary. This yields a manifold with boundary, where the boundary complex forms a 3-handle. Finally, take the cone over the boundary from an external vertex.

For the induction step, let $\D'$ be the induced complex on $V(\D)\setminus \{t\}$. Then $\D'$ is a normal pseudomanifold with boundary. Let $\t$ be an interior $3$-simplex in $\D'$. Since $Skel_{2}(\D)=Skel_{2}(\st (t,\D))$, in particular, every 2-simplex of $\D'$ is in the boundary $\p\D'$. Therefore, $t \star \tau$ is a missing 4-simplex in $\D$.

Using Lemma \ref{4-dim vertex folding}, we conclude that $\D$ was formed using either a connected sum or vertex folding at the vertex $t$ from normal $4$-pseudomanifolds. If $\D$ was formed by a connected sum of two normal $4$-pseudomanifolds, say $\D_1$ and $\D_2$, then $g_i(\D)=g_i(\D_1)+g_i(\D_2)$ for $i=2,3$. Moreover, Proposition \ref{Proposition invariant} implies that $\D_1$ and $\D_2$ are $g_2$- and $g_3$-optimal with respect to $t$. Since the number of vertices in each of $\D_1$ and $\D_2$ is less than that in $\D$, the induction hypothesis applies. On the other hand, if $\D$ was obtained by a vertex folding from a normal $4$-pseudomanifold $\D_3$ at the vertex $t$, then by Proposition \ref{Proposition invariant}, $\D_3$ is $g_2$- and $g_3$-optimal with respect to $t$. Furthermore, $g_2(\D_3)=g_2(\D)-10$. Thus, the induction hypothesis applies to the complex $\D_3$ as well.

In each intermediate step of the induction, we apply a vertex unfolding, provided the vertex $t$ remains a singular vertex in the corresponding complex. If at any stage the vertex $t$ becomes nonsingular, then the resulting complex is a $g_2$- and $g_3$-optimal $4$-manifold. Hence, by Lemma \ref{4-manifold}, it must be a stacked sphere. This completes the proof.
\hfill $\Box$

\medskip

\begin{Lemma}\label{one vertex suspension g2-minimal}
Let $M$ be a normal $3$-pseudomanifold and $\D$ be a one-vertex suspension of $M$ at a graph cone point, say $t$. If $g_3(\D)=g_3(\lk (t,\D))$, then $g_2(M)=g_2(\lk (t,M))$. 
\end{Lemma}
\begin{proof}
We prove that $f_1(M)=f_1(\st (t,M))$. Note that $\lk (t,\D)$ is isomorphic to $M$. Therefore, by the hypothesis, $g_3(\D)=g_3(M)$. Since $t$ is a graph cone point of $M$ and $\D$ is a one-vertex suspension of $M$ at $t$, the number of vertices, edges, and triangles in $\D$ in terms of those in $M$ are given by $f_0(\D)=f_0(M)+1$, $f_1(\D)=f_1(M)+f_0(M)$, and $f_2(\D)=f_2(M)+f_1(M)+|\{uv:uv\notin \st (t,M)\}|=f_2(M)+2f_1(M)-f_1(\st (t,M))$. Therefore,
 \begin{eqnarray*}
g_3(\D)&=& f_2(\D)-4f_1(\D)+10f_0(\D)-20\\
&=& f_2(M)+2f_1(M)-f_1(\st (t,M))-4f_1(M)-4f_0(M)+10f_0(M)+10-20\\
&=& [f_2(M)-3f_1(M)+6f_0(M)-10]+f_1(M)-f_1(\st (t,M))\\
&=& g_3(M)+f_1(M)-f_1(\st (t,M)).
\end{eqnarray*}
Since $g_3(\D)=g_3(M)$, it follows from the above relation that $f_1(M)=f_1(\st (t,M))$. Since $f_0(M)=f_0(\st (t,M))$, it follows that, $g_2(M)=g_2(\st (t,M))=g_2(\lk (t,M))$. This completes the proof.
\end{proof}

\begin{Lemma}\label{4 dim npm up to one-vertex suspension}
Let $K$ be a normal $4$-pseudomanifold with exactly two singularities, and suppose that $K$ is $g_2$- and $g_3$-optimal with respect to a singular vertex. Then, $K$ is obtained from the one-vertex suspension of a $g_2$-minimal normal $3$-pseudomanifold $M$ with exactly one singularity through a sequence of vertex foldings and connected sum operations. 
\end{Lemma}
\begin{proof}
Let $t$ and $t_1$ be two singular vertices in $K$, and assume that $K$ is $g_2$- and $g_3$-optimal with respect to $t$. We will prove the result through a step-by-step argument.

\begin{enumerate}
    \item [Step 1:] Using Lemma \ref{2-skeletn same}, we obtain a normal pseudomanifold $\D$ that contains the singular vertices $t$ and $t_1$ such that $Skel_{2}(\D)=Skel_{2}(\st (t,\D))$, where $K$ is obtained from $\D$ by a finite number of facet subdivisions. Moreover, $t$ and $t_1$ are the only singular vertices in $\D$. 

    \item[Step 2:] Let $\lk (t_1,\D)$ contain a $3$-simplex $xyzw$ that does not belong to $\lk (t,\D)$, where $t$ is not a vertex of $xyzw$. Since $Skel_{2}(\D)=Skel_{2}(\st (t,\D))$, we have $txyzw$ as a  missing $4$-simplex in $\D$. Therefore, by Proposition \ref{d-dim connected cum} and Lemma \ref{4-dim vertex folding}, $\D$ comes from $g_2$- and $g_3$-optimal normal $4$-pseudomanifolds by either a vertex folding or a connected sum operation, where $txyzw$ is an identified facet. We then rename each of the resultant complexes as $K$ and move to Step 1.

    \item[Step 3:] If every $3$-simplex in $\lk (t_1,\D)$ that does not contain $t$ is also in $\lk (t,\D)$, then the links of $t$ and $t_1$ are the same, since  $\D$ is closed and both $\lk (t,\D)$ and $\lk (t_1,\D)$ are closed subcomplexes of $\D$. Hence, $\D$ is a one-vertex suspension of the link of $t$ (resp. the link of $t_1$) at the vertex $t_1$ (resp. $t$). Without loss of generality, let $\D$ be a one-vertex suspension of $\lk (t,\D)$ at the vertex $t_1$. Now,  Lemma \ref{one vertex suspension g2-minimal} implies that $\lk (t,\D)$ is $g_2$ minimal with respect to $t_1$. Moreover, $\lk (t,\D)$ is a normal $3$-pseudomanifold with exactly one singularity at $t_1$. 
    
    Note that, if $t_1$ were not a singular vertex of $\lk (t,\D)$, then $\lk (t,\D)$ would be a $3$-manifold that is $g_2$-minimal, implying that $g_2(\lk (t,\D))=g_2(\lk (t_{1},\D))=0$. Consequently, $g_2(\D)=0$, and hence $\D$ is a stacked sphere. Since an application of Step 2 does not involve $t_1$, it follows that $t_1$ is a non-singular vertex in $K$, which contradicts the fact that $K$ contains two singular vertices.  
\end{enumerate}

Thus, we conclude that $\D$ is obtained by a sequence of vertex foldings and connected sums from a $g_2$- and $g_3$-optimal normal $4$-pseudomanifold $M$. Moreover, $M$ is the one-vertex suspension of a $g_2$-minimal normal $3$-pseudomanifold with exactly one singularity.
\end{proof}

We are now equipped with the essential tools to prove our main results, Theorem \ref{dim 4 2-sing main1} and Theorem \ref{dim 4 2-sing main2}.

\bigskip

\noindent {\em Proof of Theorem} \ref{dim 4 2-sing main1}.
It follows from Lemma \ref{4 dim npm up to one-vertex suspension} that $K$ is obtained from the one-vertex suspension of a $g_2$-minimal normal $3$-pseudomanifold $M$ with at most one singularity through a sequence of operations of types vertex foldings and connected sums. Since $M$ is a $g_2$-minimal normal $3$-pseudomanifold with at most one singularity, Proposition \ref{3-dim npm two singularity} implies that $M$ is obtained from some boundary complexes of $4$-simplices by a sequence of operations consisting of vertex foldings and connected sums.
\hfill $\Box$
\bigskip

\noindent {\em Proof of Theorem} \ref{dim 4 2-sing main2}. Let $t$ and $t_1$ be two singular vertices in $K$, and assume that $K$ is $g_2$- and $g_3$-optimal with respect to $t$. By Lemma \ref{4 dim npm up to one-vertex suspension}, there exists a normal $4$-pseudomanifold $\D'$ with exactly two singular vertices, $t$ and $t_1$, such that $K$ is obtained from $\D'$ by a sequence of vertex folding and connected sum operations. Moreover, $\D'$ is $g_2$- and $g_3$-optimal \wrt both $t$ and $t_1$, and it is the one-vertex suspension of $\lk (t,\D')$ at the vertex $t_1$ (resp. $\lk (t_1,\D')$ at the vertex $t$). By Lemma \ref{one vertex suspension g2-minimal}, it follows that $\lk (t,\D')$ is $g_2$-minimal \wrt $t_1$. Now, we describe the combinatorial structure of $\D'$ in terms of connected sums and edge foldings.

\begin{enumerate}
    \item [Step 1:] Using Lemma \ref{2-skeletn same}, we obtain a normal pseudomanifold $\D$ that contains the singular vertices $t$ and $t_1$ such that $Skel_{2}(\D)=Skel_{2}(\st (t,\D))$, where $\D'$ is obtained from $\D$ by a finite number of facet subdivisions. Moreover, $t$ and $t_1$ are the only two singular vertices in $\D$, and $\lk (t,\D')$ is $g_2$-minimal \wrt $t_1$ implies that $\lk (t,\D)$ is $g_2$-minimal \wrt $t_1$. Let $\t$ be a $3$-simplex in the interior of $\D[V(\D)\setminus \{t\}]$. Since $Skel_{2}(\D)=Skel_{2}(\st (t,\D))$, $t\star\t$ is a missing $4$-simplex in $\D$. Since $\D$ has exactly two singular vertices, $\t$ contains at most one singular vertex. 

    We first check if there is an interior $3$-simplex $\t$ in $\D[V(\D)\setminus \{t\}]$ that does not contain a singular vertex. If there is any such $\t$, then we move to Step 2; otherwise, we move to Step 3.
    
   \item [Step 2:] Let $V(\t)$ not contain $t_1$. Then, by Proposition \ref{d-dim connected cum} and Lemma \ref{4-dim vertex folding}, $\D$ comes from $g_2$- and $g_3$-optimal normal $4$-pseudomanifolds by either a vertex folding at $t$ or a connected sum where $t\star\t$ is the image of identified facets. If $\D$ is obtained by a vertex folding of $\D_1$ at the vertex $t$, then $g_2(\lk (t_1,\D_1))> g_2(\D_1)$, which is a contradiction. Hence, vertex folding is not possible. Therefore, $\D$ is the connected sum of two $g_2$- and $g_3$-optimal normal $4$-pseudomanifolds, say $\D_2$ and $\D_3$. Moreover, since $tt_1$ is an edge in $\D$, one of $\D_2$ or $\D_3$ must be a stacked sphere. Let $\D_2$ be the stacked sphere. Now, we rename $\D_3$ as $\D'$ and move to Step 1.

    \item [Step 3:] Let $V(\t)$ contain $t_1$. Then, the missing $4$-simplex $t\star\t$ contains two singular vertices, $t$ and $t_1$. If for at least one $y\in\{t,t_1\}$, the boundary $\p(t\star\t - y)$ separates $\lk (y,\D)$ into two parts, then by the same arguments as in Step 2, $\D$ is the connected sum of two $g_2$- and $g_3$-optimal normal $4$-pseudomanifolds. We then move to Step 1 with the resultant normal $4$-pseudomanifold that contains both singular vertices by renaming it as $\D'$. 
    
Now, assume that for each $y\in\{t,t_1\}$, the boundary $\p(t\star\t - y)$ does not separate $\lk (y,\D)$ into two parts. If, furthermore, $\lk (tt_1,\D)$ is separated into two parts by $\p(\t-t_1)$, then from Proposition \ref{d-dim connected cum} it follows that $\lk (t,\D)$ is obtained by a handle addition from a normal $3$-pseudomanifold; otherwise, if $\lk (t,\D)$ were a connected sum, then it could have been separated into two parts by $\p(t\star\t - y)$, a contradiction to the assumption. Now, by Proposition \ref{handle addition does not have g2-minimal}, a normal $3$-pseudomanifold obtained by a handle addition cannot be $g_2$-minimal, contradicting the fact that $\lk (t,\D)$ is a $g_2$-minimal normal $3$-pseudomanifold \wrt $t_1$. Therefore, $\lk (tt_1,\D)$ is not separated into two parts by $\p(\t-t_1)$.
    
    Since for each $y\in\{t,t_1\}$, $\p(t\star\t - y)$ does not separate $\lk (y,\D)$ into two parts, and $\lk (tt_1,\D)$ is not separated into two parts by $\p(\t-t_1)$, it follows from Lemma \ref{4-dim edge folding} that $\D$ is obtained from a $g_2$- and $g_3$-optimal normal $4$-pseudomanifold $\D_4$ by an edge folding along the edge $tt_1$. If $t$ (and hence $t_1$) is a singular vertex in the resultant complex $\D_4$, then we rename $\D_4$ as $\D'$ and move to Step 1. Otherwise, we proceed to Step 4.

    \item [Step 4:] Let $t$ and $t_1$ are non-singular vertices in the normal $4$-pseudomanifold $\D_4$ obtained in Step 3. Note that, $\D_4$ is $g_2$-minimal \wrt $t$, and $\lk (t,\D_4)$ is $g_2$-minimal \wrt $t_1$. Therefore, $g_2(\D_4)=g_2(\lk (t,\D_4))=g_2(\lk (tt_1,\D_4))$. Since both of $t$ and $t_1$ are non-singular vertices in $\D_4$, and $tt_1$ is an edge, $\lk (tt_1,\D_4)$ must be a $2$-sphere. Hence $g_2(\lk (tt_1,\D_4))=0$. Therefore, $g_2(\D_4)=0$.
\end{enumerate} 

From Step 4, we obtain that $\D_4$ is a stacked sphere. Therefore, $\D_4$ is obtained from the boundary complexes of $5$-simplices by a sequence of connected sum operations. Hence, combining Steps 1 to 4, we conclude that $\D'$ is obtained from the boundary complexes of $5$-simplices by a sequence of operations consisting of connected sums and edge foldings. This completes the proof. 
\hfill $\Box$

\medskip
From the constructions in Lemma \ref{4 dim npm up to one-vertex suspension}, and Theorem \ref{dim 4 2-sing main2}, we obtain a topological characterization of these normal pseudomanifolds.

\begin{Corollary}
Let \( K \) be a normal \( 4 \)-pseudomanifold with exactly two singular vertices, and assume that \( K \) is both \( g_2 \)- and \( g_3 \)-optimal with respect to a singular vertex. Then, $g_2$ satisfies \( g_2(K) = 6m + 10n \), where \( m \in \mathbb{N} \) and \( n \in \mathbb{N} \cup \{0\} \).  Furthermore, the geometric carrier \( |K| \) is PL homeomorphic to the following construction:  
\begin{itemize}
\item Start with a $3$-dimensional handlebody with exactly \( m \) handles and take its boundary coned off.  
\item Next, consider the suspension of this topological object.  
\item Finally, attach \( n \) number of $1$-handles $(\mathbb{D}_1\times \mathbb{D}_3$  that is attached along its boundary $\partial(\mathbb{D}_1)\times \mathbb{D}_3)$ to one of the boundary components and cone off both boundary components.
\end{itemize}
\end{Corollary}

\section{Future Directions}
In this work, we provide a combinatorial characterization of normal \(4\)-pseudomanifolds \(K\) with exactly two singular vertices, assuming that \(K\) is \(g_2\)- and \(g_3\)-optimal with respect to a singular vertex. We establish that such a complex \(K\) can be obtained from certain boundary complexes of \(5\)-simplices through a sequence of operations involving vertex foldings, edge foldings, and connected sums.  This result naturally leads to the following open questions:  

\begin{Question}
{\rm 
Can one provide a combinatorial characterization of normal \(4\)-pseudomanifolds \(K\) that are \(g_2\)- and \(g_3\)-optimal with respect to a singular vertex but have more than two singular vertices? }
\end{Question}

\begin{Question}
{\rm
Can a similar characterization be established for higher-dimensional normal pseudomanifolds? }
\end{Question}

Exploring these questions could lead to a deeper understanding of the combinatorial structure of pseudomanifolds and their topological properties.

\medskip

 \noindent {\bf Acknowledgement:} The first author is supported by the Mathematical Research Impact Centric Support (MATRICS) Research Grant (MTR/2022/000036) by SERB (India). The second author is supported by the institute fellowship at IIT Delhi, India.  The third author is supported by the Prime Minister's Research Fellows (PMRF/1401215) scheme.
 
\smallskip

\noindent {\bf Data availability:} The authors declare that all data supporting the findings of this study are available within the article.

\smallskip

\noindent {\bf Declarations}

\smallskip

\noindent {\bf Conflict of interest:} No potential conflict of interest was reported by the authors.

\end{document}